\documentclass[11pt]{amsart}
\usepackage{amssymb, amsfonts, amsmath}
\usepackage{enumerate}
\usepackage[mathscr]{eucal}
\usepackage{mathrsfs}
\usepackage{pgf,tikz}
  \usetikzlibrary{arrows}
\usepackage{url}
\usepackage{soul}

\newtheorem{theorem}{Theorem}
\newtheorem{corollary}[theorem]{Corollary}
\newtheorem{lemma}[theorem]{Lemma}
\newtheorem{proposition}[theorem]{Proposition}
\theoremstyle{definition}

\newtheorem{definition}[theorem]{Definition}
\newtheorem{example}[theorem]{Example}
\newtheorem{remark}[theorem]{Remark}

\begin{document}

\title{On pseudo-Frobenius elements of submonoids of $\mathbb{N}^d$}

\author{J.I. Garc\'{\i}a-Garc\'{\i}a}
\address{Universidad de C\'{a}diz\newline
\indent Departamento de Matem\'aticas/INDESS (Instituto Universitario para el Desarrollo Social Sostenible)}
\email{ignacio.garcia@uca.es}

\author{I. Ojeda}
\address{Universidad de Extremadura\newline
\indent Departamento de Matem\'aticas/IMUEx}
\email{ojedamc@unex.es}

\author{{J.C.} Rosales}
\address{Universidad de Granada\newline
\indent Departamento de \'Algebra}
\email{jcrosales@ugr.es}

\author{A. Vigneron-Tenorio}
\address{Universidad de C\'{a}diz\newline
\indent Departamento de Matem\'aticas/INDESS (Instituto Universitario para el Desarrollo Social Sostenible)}
\email{alberto.vigneron@uca.es}

\thanks{The first and the third author were partially supported by Junta de Andaluc\'{\i}a research group FQM-366 and by the project MTM2017-84890-P}

\thanks{
The second author was partially supported by the research groups FQM-024 (Junta de Extremadura/FEDER funds) and by the project MTM2015-65764-C3-1-P (MINECO/FEDER, UE) and by the project MTM2017-84890-P}

\thanks{
The fourth author was partially supported by Junta de Andaluc\'{\i}a research group FQM-366, by the project MTM2015-65764-C3-1-P (MINECO/FEDER, UE) and by the project MTM2017-84890-P}

\date{\today}

\newcommand{\comb}[2]{\left(
                      \begin{array}{c}
                       #1 \\
                       #2  \\
                      \end{array} \right)}

\begin{abstract}
In this paper we study those submonoids of $\mathbb{N}^d$ which a non-trivial pseudo-Frobenius set. In the affine case, we prove that they are the affine semigroups whose associated algebra over a field has maximal projective dimension possible. We prove that these semigroups are a natural generalization of numerical semigroups and, consequently, most of their invariants can be generalized. In the last section we introduce a new family of submonoids of $\mathbb{N}^d$ and using its pseudo-Frobenius elements we prove that the elements in the family are direct limits of affine semigroups.
\end{abstract}

\maketitle
\pagestyle{myheadings}
\markboth{{J. I.} Garc\'{\i}a-Garc\'{\i}a, I. Ojeda, {J.C.} Rosales \and {A.} Vigneron-Tenorio}{On pseudo-Frobenius elements of submonoids of $\mathbb{N}^d$}

\section*{Introduction}

Throughout this paper $\mathbb{N}$ will denote the set of nonnegative integers. Unless otherwise stated, all considered semigroups will be submonoids of $\mathbb{N}^d$.
Finetelly generated submonoids of $\mathbb{N}^d$ will be called, affine semigroups, as ussual. If $d=1$, affine semigroups are called numerical semigroups. Numerical semigroup has been widely studied in the literature (see, for instance, \cite{ns-book} and the references therein). It is well-known that $S \subseteq \mathbb{N}$ is a numerical semigroup if and only if $S$ is a submonoid of $\mathbb{N}$ such that $\mathbb{N} \setminus S$ is a finite set. Clearly, this does not hold for affine semigroups in general. In \cite{Csemig1}, the affine semigroups whose complementary in the cone that they generate is finite, in a suitable sense, are called $\mathcal{C}-$semigroups and the authors prove that a Wilf's conjecture can be generalized to this new situation.

In the numerical case, the finiteness of $\mathbb{N} \setminus S$ implies that there exists at least a positive integers $a \in \mathbb{N} \setminus S$ such that $a+S \setminus \{0\} \subseteq S$ (provided that $S \neq \mathbb{N}$). These integers are called pseudo-Frobenius numbers and the biggest one is the so-called Frobenius number. In this paper, we consider the semigroups such that there exists at least one element $\mathbf{a} \in \mathbb{N}^d \setminus S$ with $\mathbf a+S \setminus \{0\} \subseteq S$. By analogy, we call these elements the pseudo-Frobenius elements of $S$. We emphasize that the semigroups in the family of $\mathcal{C}-$semigroups have pseudo-Frobenius elements.

One the main results in this paper is Theorem \ref{Th1} which states that an affine semigroup, $S$, has pseudo-Frobenius elements if and only if the length of the minimal free resolution of the semigroup algebra $\Bbbk[S],$ with $\Bbbk$ being a field, as a module over a polynomial ring is maximal, that is, if $\Bbbk[S]$ has the maximal projective dimension possible (see Section \ref{S2}). For this reason we will say that affine semigroups with pseudo-Frobenius elements are maximal projective dimension semigroups, MPD-semigroups for short. Observe that, as minimal free resolution of semigroup algebras can be effectively computed, our results provides a computable necessary and sufficient condition for an affine semigroup to be a MPD-semigroup. In fact, using the procedure outlined \cite{OjVi}, one could theoretically compute the minimal free resolution of $\Bbbk[S]$ from the pseudo-Frobenius elements. To this end, the effective computation of the pseudo-Frobenius is needed. We provide a bound of the pseudo-Frobenius elements in terms of $d$ and the cardinal and size of the minimal generating set of $S$ (see Corollary \ref{Cor bound}).

In order to emulate the maximal property of the Frobenius element of a numerical semigroup, we fix a term order on $\mathbb{N}^d$ and define Frobenius elements as the maximal elements for the fixed order of the elements in the integral points in the complementary of an affine semigroup $S$ with respect to its cone. These Frobenius elements (if exist) are necesarily pseudo-Frobenius elements of $S$ for a maximality matter (see Lemma \ref{lemma1}). So, if Frobenius elements of $S$ exist then $\Bbbk[S]$ has maximal projective dimension, unfortunately the converse it is not true in general. However, there are relevant families of affine semigroups with Frobenius elements as the family of $\mathcal{C}-$semigroups. In the section devoted to Frobenius elements (Section \ref{S3}), we prove generalizations of well-known results for numerical semigroups such as Selmer's Theorem that relates the Frobenius numbers and Ap\'ery sets (Theorem \ref{Th Selmer}) or the characterization of the pseudo Frobenius elements in terms of the Ap\'ery sets (Proposition \ref{Prop ApPF}). We close this section proving that affine semigroups having Frobenius elements are stable by gluing, by giving a formula for a Frobenius element in the gluing (Theorem \ref{frob-gluing}).

In the Section \ref{Sect Irr}, we deal with the problem of the irreducibility of the MPD-semigroups. Again, we prove the analogous results for MPD-semigroups than the known-ones for irreducible numerical semigroups. Of special interest is the characterization of the pseudo-Frobenius sets for $\mathcal{C}-$semigroups (Theorem \ref{Th Irr} and Proposition \ref{Prop IrredC}). Pedro A. Garc\'{\i}a-S\'anchez communicated us that similar results were obtained by C. Cisto, G. Fiolla, C. Peterson and R. Utano in \cite{CFPU} for $\mathbb{N}^d-$semigroups, that is, those $\mathcal{C}-$semigroups whose associated cone is the whole $\mathbb{N}^d$.

Finally, in the last section of this paper, we introduce a new family of (non-necessarily finitely generated) submonoids of $\mathbb{N}^d$ which have pseudo-Frobenius elements, the elements of the family are called PI-monoids. These monoids are a natural generalization of the MED-semigroups (see \cite[Chapter 3]{ns-book}). We conclude the paper by proving that any PI-monoid is direct limit of MPD-semigroups.

The study of $\mathbb{N}^d-$semigroups, which is a subfamily of the MPD-semigroups, is becoming to be an active research area in affine semigroups theory. For instance, in \cite{CFU} algorithms for dealing with $\mathbb{N}^d-$semigroups are given. These algorithms are implemented in the development version site of the GAP (\cite{GAP}) \texttt{NumericalSgps}  package (\cite{numericalsgps}):

\centerline{\url{https://github.com/gap-packages/numericalsgps}.}

\section{Pseudo-Frobenius elements of affine semigroups}\label{S1}

Set $\mathcal{A} = \{\mathbf{a}_1, \ldots, \mathbf{a}_n\} \subset \mathbb{N}^d$ and let $S$ be the submonoid of $\mathbb{N}^d$ generated by $\mathcal{A}$. Consider the cone of $S$ in $\mathbb{Q}^d_{\geq 0}$ $$\mathrm{pos}(S) := \left\{\sum_{i=1}^n \lambda_i \mathbf{a}_i\ \mid\ \lambda_i \in \mathbb{Q}_{\geq 0},\ i = 1, \ldots,n \right\}$$
and define $\mathcal{H}(S) := (\mathrm{pos}(S) \setminus S) \cap \mathbb{N}^d.$

\begin{definition}
An integer vector $\mathbf{a} \in \mathcal{H}(S)$ is called a \textbf{pseudo-Frobenius element} of $S$ if $\mathbf a + S\setminus \{0\} \subseteq S.$ The set of pseudo-Frobenius elements of $S$ is denoted by $\textrm{PF}(S)$.
\end{definition}

Observe that the set $\mathrm{PF}(S)$ may be empty: indeed, let $$\mathcal{A}=\{(2,0),(1,1),(0,2)\} \subset \mathbb{N}^2.$$ The semigroup $S$ generated by $\mathcal{A}$ is the subset of points in $\mathbb{N}^2$ whose sum of coordinates is even. Thus, we has that $\mathcal{H}(S) + S = \mathcal H (S)$. Therefore $\mathrm{PF}(S) = \varnothing$.

On other hand, $\mathrm{PF}(S) \neq \mathcal{H}(S)$ in general as the following example shows.

\begin{example}\label{Ex Irr?}
Let $S$ be the submonoid of $\mathbb{N}^2$ generated by the columns of the following matrix
$$
A =
\left(
\begin{array}{ccccc}
3 & 5 & 0 & 1 & 2 \\ 0 & 0 & 1 & 3 & 3
\end{array}
\right)
$$
In this case, \begin{align*}
\mathcal{H}(S) = \{ &  (1,0),(2,0),(4,0),(1,1),(2,1),(4,1), \\ & (1,2),(2,2),(4,2),(7,0),(7,1),(7,2)\},
\end{align*}
whereas $\mathrm{PF}(S) = \{(7,2)\}:$
\begin{center}
\begin{tikzpicture}[line cap=round,line join=round,>=triangle 45,x=1.0cm,y=1.0cm,scale=.75]
\draw [color=gray,dash pattern=on 2pt off 2pt, xstep=1.0cm,ystep=1.0cm] (0,0) grid (9.5,4.5);
\draw[->,color=black] (0,0) -- (9.5,0);
\foreach \x in {1,2,3,4,5,6,7,8}
\draw[shift={(\x,0)},color=black] (0pt,2pt) -- (0pt,-2pt) node[below] {\footnotesize $\x$};
\draw[->,color=black] (0,0) -- (0,4.5);
\foreach \y in {1,2,3,4}
\draw[shift={(0,\y)},color=black] (2pt,0pt) -- (-2pt,0pt) node[left] {\footnotesize $\y$};
\draw[color=black] (0pt,-10pt) node[right] {\footnotesize $0$};
\clip(-0.5,-0.5) rectangle (9.5,4.5);
\fill[fill=blue,fill opacity=0.1] (0,3) -- (10,3) -- (10,5) -- (0,5) -- cycle;
\fill[fill=blue,fill opacity=0.1] (8,0) -- (10,0) -- (10,3) -- (8,3) -- cycle;
\begin{scriptsize}
\fill [color=blue] (3,0) circle (3pt);
\draw[color=blue] (3.3,0.26) node {$a_1$};
\fill [color=blue] (5,0) circle (3pt);
\fill [color=blue] (5,1) circle (2pt);
\fill [color=blue] (5,2) circle (2pt);
\fill [color=blue] (6,0) circle (2pt);
\fill [color=blue] (6,1) circle (2pt);
\fill [color=blue] (6,2) circle (2pt);
\draw[color=blue] (5.3,0.26) node {$a_2$};
\fill [color=blue] (0,1) circle (3pt);
\draw[color=blue] (0.3,1.26) node {$a_3$};
\fill [color=blue] (1,3) circle (3pt);
\draw[color=blue] (1.3,3.26) node {$a_4$};
\fill [color=blue] (2,3) circle (3pt);
\draw[color=blue] (2.3,3.26) node {$a_5$};
\fill [color=blue] (0,0) circle (2pt);
\fill [color=blue] (0,2) circle (2pt);
\fill [color=blue] (3,1) circle (2pt);
\fill [color=blue] (3,2) circle (2pt);
\draw[->,color=red] (0,0) -- (7,2);
\draw[color=red] (4,1.33) node {$\mathbf v$};
\end{scriptsize}
\end{tikzpicture}
\end{center}
The elements in $S$ are the blue points and the points in the shadowed blue area. The red vector is the only pseudo-Frobenius element of $S$. The big blue points correspond to the minimal generators of $S$.
\end{example}

There are relevant families of numerical semigroups for which the pseudo-Frobenius element exist.

\begin{definition}
If $\mathcal{H}(S)$ is finite, then $S$ is said to be a \textbf{$\mathcal{C}-$semigroup}, where $\mathcal{C}$ denotes the cone $\mathrm{pos}(S)$.
\end{definition}

The $\mathcal{C}-$semigroups were introduced in \cite{Csemig1}.

\begin{proposition}
If $S$ is a $\mathcal{C}-$semigroup different from $\mathcal{C}\cap \mathbb{N}^d$, then $\textrm{PF}(S) \neq \varnothing$.
\end{proposition}

\begin{proof}
Let $\preceq$ be a term order on $\mathbb{N}^d$ and set $\mathbf{a}:=max_{\preceq}(\mathcal{H}(S))$. If $\mathbf{b} \in S \setminus \{0\}$ is such that $\mathbf{a} + \mathbf{b} \not\in S$, then $\mathbf{a} + \mathbf{b} \in \mathcal{H}(S)$ and $\mathbf{a} + \mathbf{b} \succ \mathbf{a}$ which contradicts the maximality of $\mathbf{a}$.
\end{proof}

The converse of the above proposition is not true, as the following example shows.

\begin{example}\label{Ex no C-semig}
Let $\mathcal{A} \subset \mathbb{N}^2$ be the columns of the matrix
\[
A = \left(\begin{array}{cccccccccccc}
18 & 18 & 4 & 20 & 23 & 8 & 11 & 11 & 10 & 14 & 7 & 7 \\
 9 &  3 & 1 &  8 & 10 & 3 &  5 &  2 &  3 &  3 & 2 & 3
\end{array}\right)
\]
and let $S$ be the subsemigroup of $\mathbb{N}^2$ generated by $\mathcal{A}$.
The elements in $S$ are the integer points in an infinite family of homotetic pentagons.
\begin{center}
\includegraphics[scale=.6]{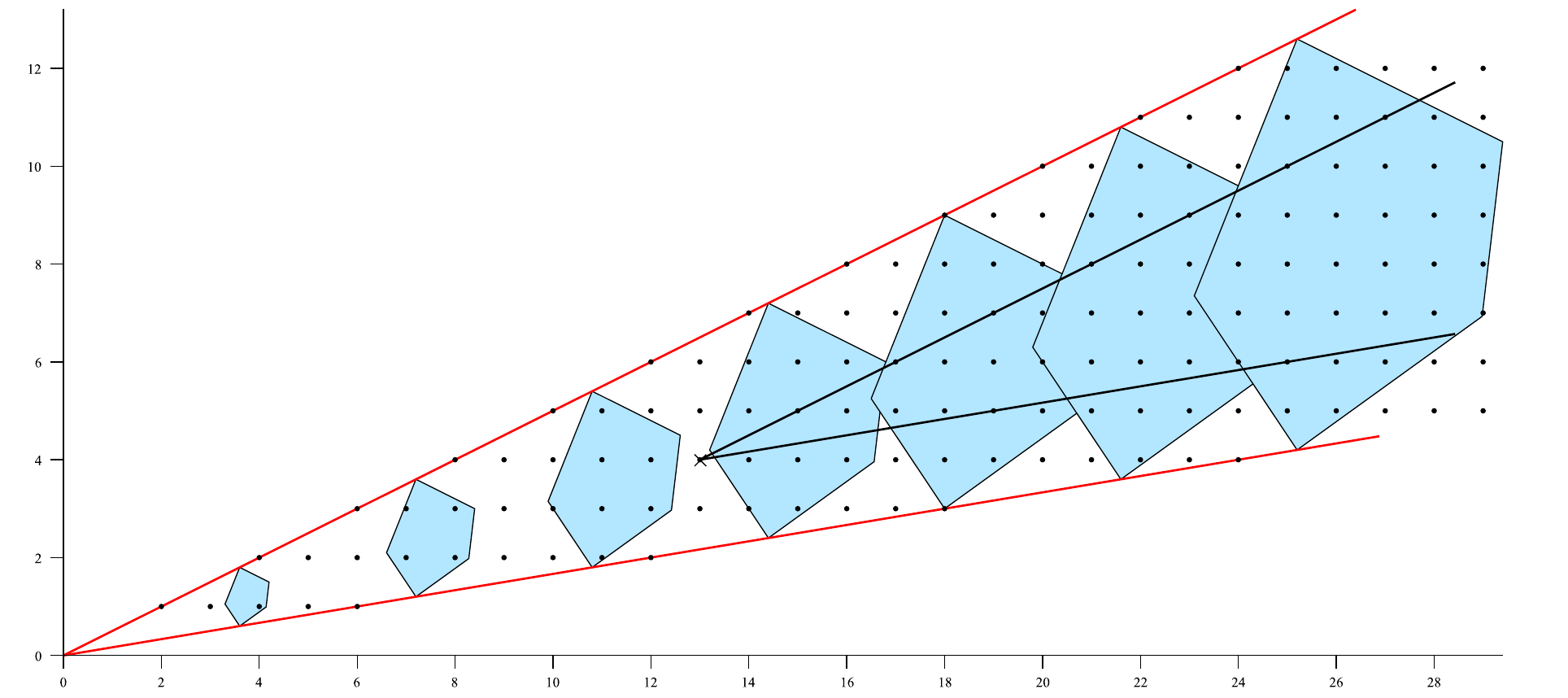}
\end{center}
This semigroup is a so-called multiple convex body semigroup (see \cite{multiple} for further details). Clearly, $S$ is not a $\mathcal{C}-$semigroup, but $(13,4)\in \textrm{PF}(S)$.

\end{example}

\section{Maximal projective dimension}\label{S2}

As in the previous section, set $\mathcal{A} = \{\mathbf{a}_1, \ldots, \mathbf{a}_n\} \subset \mathbb{N}^d$ and let $S$ be the submonoid of $\mathbb{N}^d$ generated by $\mathcal{A}$. Let $\Bbbk$ be an arbitrary field.

The surjective $\Bbbk$-algebra morphism $$\varphi_0 : R:=\Bbbk[x_1, \ldots, x_n]
\longrightarrow \Bbbk[S] := \bigoplus_{\mathbf a \in S} \Bbbk\, \chi^\mathbf{a};\ x_i \longmapsto \chi^{\mathbf{a}_i}$$ is
$S$-graded, thus, the ideal $I_S := \ker(\varphi_0)$ is a
$S$-homogeneous ideal called the ideal of $S.$ Notice that $I_S$ is
a toric ideal generated by $$\Big\{\mathbf{x}^\mathbf{u} - \mathbf{x}^\mathbf{v} :
\sum_{i=1}^n u_i \mathbf{a}_i = \sum_{i=1}^n v_i \mathbf{a}_i \Big \}.$$

Now, by using the $S$-graded Nakayama's lemma recursively, we may construct $S-$graded
$\Bbbk-$algebra homomorphism $$\varphi_{j+1} : R^{s_{j+1}}
\longrightarrow R^{s_j},$$ corresponding to a choice of a minimal
set of $S-$homogeneous generators for each module of syzygies $N_j
:= \ker(\varphi_j),j \geq 0$ (see \cite{BCMP} and the references therein). Notice that $N_0 = I_S.$ Thus, we obtain a
minimal free $S-$graded resolution for the $R-$module
$\Bbbk[S]$ of the form
$$
\ldots \longrightarrow R^{s_{j+1}}
\stackrel{\varphi_{j+1}}{\longrightarrow} R^{s_j} \longrightarrow
\ldots \longrightarrow R^{s_2} \stackrel{\varphi_2}{\longrightarrow}
R^{s_1} \stackrel{\varphi_1}{\longrightarrow} R
\stackrel{\varphi_{0}}{\longrightarrow} \Bbbk[S]
\longrightarrow 0,
$$
where $s_{j+1} := \sum_{\mathbf b \in S} \mathrm{dim}_\Bbbk V_j(\mathbf b),$
with $V_j(\mathbf b) := (N_j)_{\mathbf b}/(\mathfrak{m} N_j)_{\mathbf b},$ is the so-called
$(j+1)$th Betti number of $\Bbbk[S]$, where $\mathfrak{m}= \langle x_1, \ldots, x_n \rangle$ is the irrelevant maximal ideal.

Observe that the dimension of $V_j(\mathbf b)$ is
the number of generators of degree ${\mathbf b}$ in a minimal system of
generators of the $j$th module of syzygies $N_j$ (i.e. the
multigraded Betti number $s_{j,\mathbf  b}$). The $\mathbf{b} \in S$ such that $V_j(\mathbf b) \neq 0$ are called \textbf{$S-$degrees of the $j-$minimal syzygy} of $\Bbbk[S]$. So, by the Noetherian property
of $R,\ s_{j+1}$ is finite. Moreover, by the Hilbert's syzygy theorem and the Auslander-Buchsbaum's
formula, it follows that $s_j = 0$ for $j > p = n - \mathrm{depth}_R
\Bbbk[S]$ and $s_p \neq 0$. Such integer $p$ is called the \textbf{projective dimension of} $S$.

Since $\mathrm{depth}_R \Bbbk[S] \geq 1$, the projective dimension of $S$ is lesser than or equal to $n-1$.
We will say that $S$ is a \textbf{maximal projective dimension semigroup} (MPD-semigroup, for short) if its projective dimension is $n-1$, equivalently, if $\mathrm{depth}_R \Bbbk[S] = 1$.

Recall that $S$ is said to be Cohen-Macaulay if $\mathrm{depth}_R \Bbbk[S] = \dim(\Bbbk[S])$. So, if $S$ is a MPD-semigroup, then $S$ is Cohen-Macaulay if and only if $\Bbbk[S]$ is the coordinate ring of a monomial curve; equivalently, $S$ is a numerical semigroup.

\begin{theorem}\label{Th1}
The necessary and sufficient condition for $S$ to be a MPD-semigroup is that $\mathrm{PF}(S) \neq \varnothing$.
In this case, $\mathrm{PF}(S)$ has finite cardinality.
\end{theorem}

\begin{proof}
By definition, $S$ is a MPD-semigroup if and only if $V_{n-2}(\mathbf{b}) \neq \varnothing$ for some $\mathbf{b} \in S$. By \cite[Theorem 2.1]{BCMP}, given $\mathbf{b} \in S$, $V_{n-2}(\mathbf{b}) \neq \varnothing$ if and only if $\mathbf{b} - \sum_{i=1}^n \mathbf{a}_i \not\in S$ and $\mathbf{b} - \sum_{i \in F} \mathbf{a}_i \in S,$ for every $F \subsetneq \{1, \ldots, n\}$. Clearly, if $\mathrm{PF}(S) \neq \varnothing$, there exists $\mathbf{a} \in \mathcal{H}(S)$ such that $\mathbf{a} + S \setminus \{0\} \subseteq S$. So, by taking $\mathbf{b} = \mathbf{a} + \sum_{i=1}^n \mathbf a_i$, one has that $\mathbf{b} - \sum_{i=1}^n \mathbf{a}_i \not\in S$ and $\mathbf{b} - \sum_{i \in F} \mathbf{a}_i \in S,$ for every $F \subsetneq \{1, \ldots, n\}$, and we conclude $S$ is a MPD-semigroup. Conversely, if there exists $\mathbf{b} \in S$ such that $V_{n-2}(\mathbf{b}) \neq \varnothing$, then we have that $\mathbf{a} := \mathbf{b} - \sum_{i=1}^n \mathbf a_i \in \mathbb{Z}^d \setminus S$. Clearly,
$\mathbf{a} \in \mathbb{Z}^d \setminus S$ and $\mathbf{a} + S \setminus \{0\} \subseteq S$, because $\mathbf{a} + \mathbf{a}_j \in S,$ for every $j \in \{1, \ldots, n\}.$ Thus,  in order to see that $\mathbf{a} \in \mathcal{H}(S)$ it suffices to prove that $\mathbf{a} \in \mathrm{pos}(S)$. Without loss of generality, we assume that $\{a_1, \ldots, a_\ell\}$ is a minimal set of generators of $\mathrm{pos}(S)$.
By Farkas' Lemma, $\mathrm{pos}(S)$ is a rational convex polyhedral cone. Then, for each $j \in \{1, \ldots, \ell\}$ there exists $\mathbf{c}_j \in \mathbb{R}^d$ such that $\mathbf{a}_j \cdot \mathbf{c}_j = 0$ and
$\mathbf{a}_i \cdot \mathbf{c}_j \geq 0,\ i \neq j,$ where $\cdot$ denotes the usual inner product on $\mathbb{R}^n$. Now, since $\mathbf{a} + \mathbf{a}_j \in S,$ one has that $\mathbf{a} + \mathbf{a}_j = \sum_{i=1}^n u_{ij} \mathbf{a}_i$, for some $u_{ij} \in \mathbb{N}$. Therefore \[\mathbf{a} \cdot \mathbf{c}_j = (\mathbf{a} + \mathbf{a}_j) \cdot \mathbf c_j = (\sum_{i=1}^n u_{ij} \mathbf{a}_i) \cdot \mathbf{c}_j = \sum_{i=1}^n u_{ij} (\mathbf{a}_i \cdot \mathbf{c}_j) \geq 0,\] for every $j \in \{1, \ldots, \ell\}$. That is to say $\mathbf{a} \in \mathrm{pos}(S)$.

Finally, if $\mathrm{PF}(S) \neq \varnothing$, the finiteness of $\mathrm{PF}(S)$ follows from the finiteness of the Betti numbers.
\end{proof}

Observe that from the arguments in the proof of Theorem \ref{Th1} it follows that $s_p$ is the cardinality of $\mathrm{PF}(S)$. In fact, we have proved the following fact:

\begin{corollary}\label{cor1}
If $S$ is a MPD-semigroup, then $\mathbf{b} \in S$ is the $S-$degree of the $(n-2)$th minimal syzygy of $\Bbbk[S]$ if and only if $\mathbf{b} \in \{\mathbf{a} + \sum_{i=1}^n \mathbf{a}_i,\ \mathbf{a} \in \mathrm{PF}(S)\}$.
\end{corollary}

In \cite{OjVi}, it is outlined a procedure for a partial computation of the minimal free resolution of $S$ starting from a set of $S-$degrees of the $j$th minimal syzygy of $\Bbbk[S],$ for some $j.$ This procedure gives a whole free resolution if one knows
all the $S-$degrees of the $p$th minimal syzygies of $\Bbbk[S]$, where $p$ is the projective dimension of $S$. Therefore, by Corollary \ref{cor1}, if $S$ is a MPD-semigroup, we can use the proposed method in \cite{OjVi} to compute the minimal free resolution of $\Bbbk[S]$, provided that we were able to compute $\mathrm{PF}(S)$. However, this is not easy at all, for this reason it is highly interesting to given bounds for the elements in $\mathrm{PF}(S)$.

Given $\mathbf u = (u_1, \ldots, u_n) \in \mathbb{N}^n$, let $\ell(u)$ be the length of $\mathbf{u}$ that is, $\ell(\mathbf{u}) = \sum_{i=1}^n u_i$, and for an $d \times n-$integer matrix $B=(\mathbf b_1|\cdots | \mathbf  b_n)$, we will write $\vert \vert B \vert \vert _{\infty}$ for $\mbox{max}_i \sum_{j=1}^n |b_{ij}|$. In \cite{BPV}, the authors provide an explicit bound for the $S-$degrees of the minimal generators of $N_j$, for every $j \in \{1, \ldots, p\}$.

Let $A \in \mathbb{N}^{d \times n}$ be the matrix whose $i$th column is $\mathbf{a}_i,\ i = 1, \ldots, n$.

\begin{theorem}{\cite[Theorem 3.2]{BPV}} \label{t34}
If $\mathbf{b} \in S$ is an $S-$degree of a minimal $j-$syzygy of $\Bbbk[S]$, then $\mathbf{b} = A \mathbf{u}$ with $\mathbf{u} \in \mathbb{N}^n$ such that $$\ell(\mathbf{u}) \leq (1 + 4\, \vert \vert A \vert \vert _{\infty})^{d(d_j-1)} + (j+1)d_j -1,$$ where $d_j = \comb {n}{j+1}$.
\end{theorem}

By Corollary \ref{cor1}, this bound can be particularized for $j=n-2$ as follows.

\begin{corollary}\label{Cor bound}
Let $S$ be a MPD-semigroup. If $\mathbf{a} \in \mathrm{PF}(S)$, then $\mathbf{a} = A \mathbf{v}$ for some $\mathbf{v} \in \mathbb{N}^n$ satisfying $$\ell(\mathbf v) \leq (1 + 4\, \vert \vert A \vert \vert _{\infty})^{d(n-1)} + n^2-1.$$
\end{corollary}

\begin{proof}
Assuming $\mathbf{a}$ is a pseudo-Frobenius element of $S$, by Corollary \ref{cor1}, there exists $\mathbf{b}=A \mathbf u \in S$ for some $ \mathbf u\in\mathbb{N}^n,$ with $V_{n-2}(\mathbf{b}) \neq \varnothing$ and such that $\mathbf{a}=\mathbf{b}-\sum_{i=1}^n \mathbf a_i$; in particular $\mathbf{a}= A \mathbf{u}-A\, \mathbf{1}=A(\mathbf{u}-\mathbf{1})$. Consider $\mathbf{v}=\mathbf{u}-\mathbf{1}$. Note that $\vert \vert \mathbf{v} \vert \vert _1\le \vert \vert \mathbf{u} \vert \vert _1 + \vert \vert \mathbf{1} \vert \vert _1= \vert \vert \mathbf{u} \vert \vert _1 +n$. By Theorem \ref{t34}, $\vert \vert \mathbf{u} \vert \vert _1\le (1 + 4 \vert \vert A \vert \vert _{\infty})^{d(d_j-1)} + (j+1)d_j -1$ where $j=n-2$ and $d_{n-2}=n$. So,
\begin{align*}
\ell(\mathbf{v}) & \le  (1 + 4\, \vert \vert A \vert \vert _{\infty})^{d(n-1)} + (n-1)n-1+n\\ & = (1 + 4 \vert \vert A \vert \vert _{\infty})^{d(n-1)} + n^2-1,
\end{align*}
as claimed.
\end{proof}

Note that, given any affine semigroup and the graded minimal free resolution of its associated algebra over a field, Theorem \ref{Th1} and Corollary \ref{cor1} allow us to check if the semigroup has pseudo-Frobenius elements and, in affirmative case, to compute them. Thus, the combination of both results provide an algorithm for the computation of the pseudo-Frobenius elements of a affine semigroup, provided that they exist (i.e. if the depth of the algebra is one). The following example illustrates this fact:

\begin{example}
Let $S$ be the multiple convex body semigroup associated to the convex hull $\mathcal{P}$ of the set $\{(1.2, .35), (1.4, 0), (1.5, 0), (1.4, 1)\}$, that is to say, \[S=\displaystyle{\bigcup_{k \in \mathbb{N}}\, k\, \mathcal{P}\cap \mathbb{N}^2}.\]
\begin{center}
\includegraphics[scale=.38]{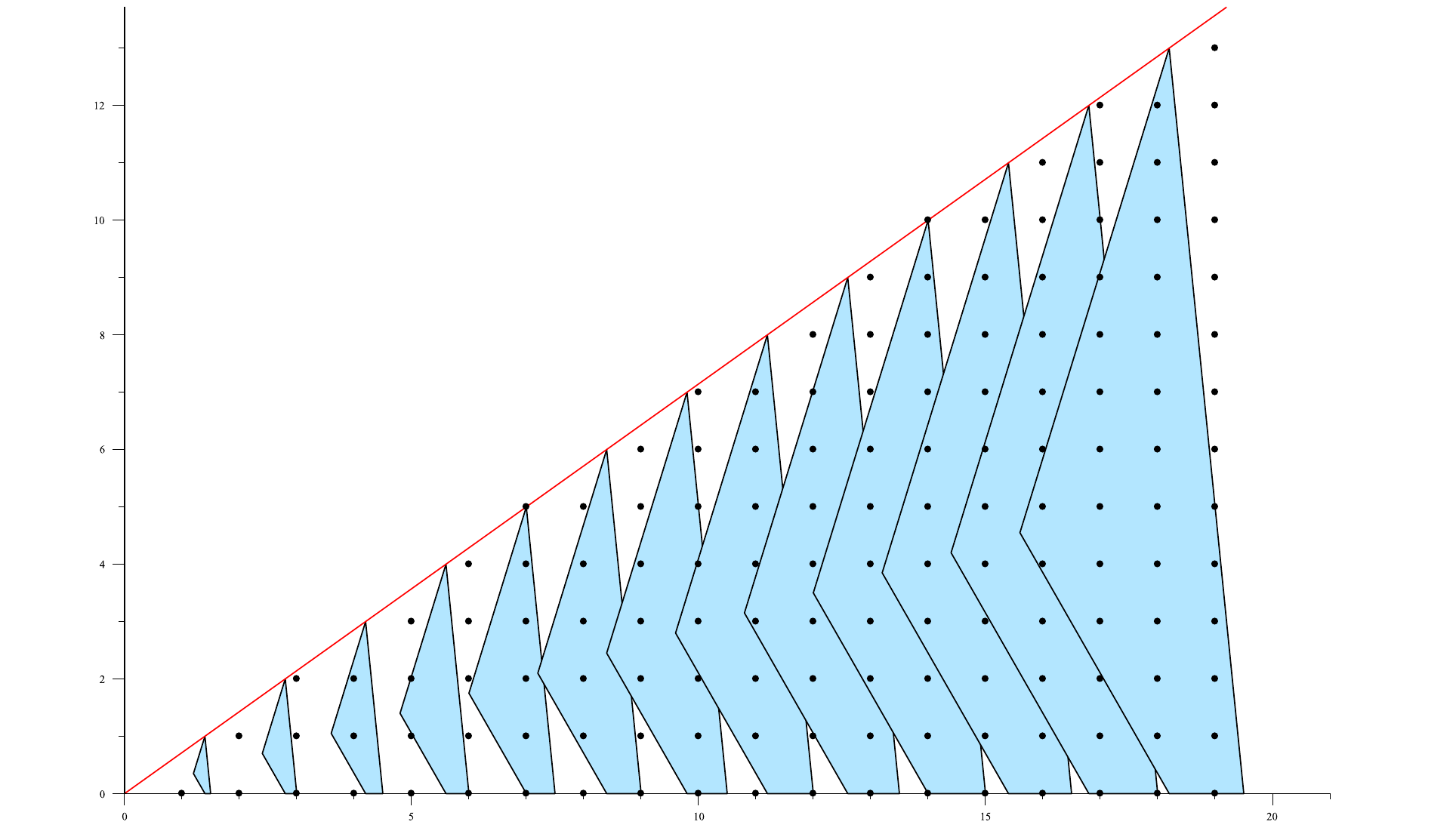}
\end{center}
Using the \texttt{Mathematica} package \texttt{PolySGTools} introduced in \cite{multiple}, we obtain that the minimal generating system of $S$ is the set of columns of the following matrix
\[
A=\left(
\begin{array}{cccccccccc}
 3 & 4 & 4 & 5 & 7 & 7 & 7 & 7 & 8 & 9 \\
 0 & 1 & 2 & 2 & 0 & 3 & 4 & 5 & 1 & 2 \\
\end{array}
\right).
\]
Now, we can easily check that $S$ is not a $\mathcal{C}-$semigroup, because $(3,2)+\lambda (7,5) \in \mathcal{H}(S),$ for every $\lambda \in \mathbb{N}$. Moreover, we can compute the $S-$graded minimal free resolution of $\Bbbk[S]$ using Singular (\cite{DGPS}) as follows:
\begin{verbatim}
   LIB "toric.lib";
   LIB "multigrading.lib";
   ring r = 0, (x(1..10)), dp;
   intmat A[2][10] = 3, 4, 4, 5, 7, 7, 7, 7, 8, 9,
                     0, 1, 2, 2, 0, 3, 4, 5, 1, 2;
   setBaseMultigrading(A);
   ideal i = toric_ideal(A,"ect");
   def L = multiDegResolution(i,9,0);
\end{verbatim}
Finally, using the command \texttt{multiDeg(L[9])}, we obtain that the degrees of the minimal generators of the $9-$th syzygy module are $(72,20)$ and $(73,21)$. So, $S$ has two pseudo-Frobenius elements: $(11,0)=(72,20)- (61,20)$ and $(12,1)=(73,21)- (61,20)$.
\end{example}

\section{On Frobenius elements of MPD-semigroups}\label{S3}

Throughout this section, $S$ will be a MPD-semigroup generated by $\mathcal{A} = \{\mathbf{a}_1, \ldots, \mathbf{a}_n\} \subset \mathbb{N}^d$.

\begin{definition}
We say that $\mathbf{f} \in \mathcal{H}(S)$ is a \textbf{Frobenius element} of $S$ if
$\mathbf{f} = \max_{\prec} \mathcal{H}(S)$ for some term order $\prec$ on $\mathbb{N}^d.$ Let us write $\mathrm{F}(S)$ for the set of Frobenius elements of $S$.
\end{definition}

Frobenius elements of $S$ may not exist. However, if $S$ is a $\mathcal C-$semigroup, then it has Frobenius elements because $\mathcal{H}(S)$ is finite.

\begin{lemma}\label{lemma1}
Every Frobenius element of $S$ is a pseudo-Frobenius element of $S,$ in symbols: $\mathrm{F}(S) \subseteq \mathrm{PF}(S)$.
\end{lemma}

\begin{proof}
If $\mathbf{f} \in \mathrm{F}(S)$, then there is a term order $\prec$ on $\mathbb{N}^d$ such that $\mathbf f = \max_{\prec} \mathcal{H}(S)$. If there exists $\mathbf{a} \in S \setminus \{0\}$ such that $\mathbf f + \mathbf a \not\in S$, then $\mathbf f \prec \mathbf f + \mathbf a \in \mathcal{H}(S)$, in contradiction to the maximality of $\mathbf{f}$. Therefore $\mathbf{f} \in \mathrm{PF}(S)$.
\end{proof}

The following notion of Frobenius vectors was introduced in \cite{AGSO}: we say that $S$ has a \textbf{Frobenius vector} if there exists $\mathbf f\in G(\mathcal{A})\setminus S$ such that \[\mathbf f+ \mathrm{relint}(\mathrm{pos}(S)) \cap G(\mathcal{A})\subseteq S\setminus\{0\}\subseteq S,\] where $G(\mathcal{A})$ denotes the group generated by $\mathcal{A}$ in $\mathbb{Z}^d$ and $\mathrm{relint}(\mathrm{pos}(S))$ the relative interior of the cone $\mathrm{pos}(S)$.

\begin{proposition}
Every Frobenius element of $S$ is a Frobenius vector of $S$.
\end{proposition}

\begin{proof}
Let $\mathbf f \in \mathrm{F}(S)$. Since $\mathbf{a} := \mathbf f + \mathbf a_1 \in S$, then $\mathbf{f} = \mathbf{a} - \mathbf{a}_1 \in G(\mathcal{A}) \setminus S$. Let $\mathbf{b} \in \mathrm{relint}(\mathrm{pos}(S)) \cap G(\mathcal{A})$. If $\mathbf{b} \in S$,  then  $\mathbf{f}+\mathbf{b} \in S$ by Lemma \ref{lemma1}. If $\mathbf{b} \not\in S$,  then $\mathbf{b} \in \mathcal{H}(S) = (\mathrm{pos}(S) \setminus S) \cap \mathbb{N}^d$ and therefore either $\mathbf{f}+\mathbf{b} \in S$ or $\mathbf{f}+\mathbf{b} \in \mathcal{H}(S)$. However, since $\mathbf{f}+\mathbf{b}$ is greater than $\mathbf{f}$ for every term order on $\mathbb{N}^d$, we are done.
\end{proof}

As a consequence of the above result, we have that the set of $\mathcal{C}$-semigroups is a new family of affine semigroups for which Frobenius vectors exist.

Although Frobenius vectors may not exist in general, there are families of submonoids of $\mathbb{N}^d$ with Frobenius vectors that are not MPD-semigroups (see \cite{AGSO}). However, even for MPD-semigroups, the converse of the above proposition does not hold in general, for instance, the MPD-semigroup in Example \ref{Ex no C-semig} has a Frobenius vector which is not a Frobenius element.

Now, similarly to the numerical case ($n=1$), we can give a Selmer's formula type (see \cite[Proposition 2.12(a)]{ns-book}) for $\mathcal{C}-$semigroups. To do this, we need to recall the notion of Ap\'ery set.

\begin{definition}\label{Def Apery}
The \textbf{Ap\'ery set} of a submonoid $S$ of $\mathbb{N}^d$ relative to $\mathbf{b} \in S \setminus \{0\}$ is defined as $\mathrm{Ap}(S,\mathbf{b}) = \{\mathbf{a} \in S\ \mid\ \mathbf{a}-\mathbf{b} \in \mathrm{pos}(S) \setminus S\}.$
\end{definition}

Clearly $\mathrm{Ap}(S,\mathbf{b}) - \mathbf{b} \subseteq \mathcal{H}(S)$; in particular, if $S$ is $\mathcal{C}-$semigroup, we have that $\mathrm{Ap}(S,\mathbf{b})$ is finite for every $\mathbf{b} \in S\setminus\{0\}$.

\begin{proposition}\label{Prop6-Cor7}
Let $S$ be a submonoid of $\mathbb{N}^d$ and $\mathbf b \in S \setminus \{0\}$. For each $\mathbf a \in S$ there exists an unique $(k, \mathbf{c}) \in \mathbb N \times \mathrm{Ap}(S, \mathbf b)$ such that $\mathbf a = k\, \mathbf b + \mathbf c$. In particular, $\big(\mathrm{Ap}(S, \mathbf b) \setminus \{0\} \big) \cup \{\mathbf b\}$ is a system of generators of $S$.
\end{proposition}

\begin{proof}
If suffices to take $k$ as the highest non-negative integer such that $\mathbf b - k \mathbf a \in S$.
\end{proof}

The following result is the generalization of \cite[Proposition 2.12(a)]{ns-book}.

\begin{theorem}\label{Th Selmer}
If $\mathbf{f} \in \mathrm{F}(S)$, there exists a term order $\prec$ on $\mathbb{N}^d$ such that \[\mathbf f = \max_{\prec} \mathrm{Ap}(S,\mathbf{b}) - \mathbf{b},\] for every $\mathbf{b} \in S \setminus \{0\}$.
\end{theorem}

\begin{proof}
By definition, there exists a term order $\prec$ on $\mathbb{N}^d$ such that $\mathbf{f} =  \max_\prec \mathcal{H}(S)$. By Lemma \ref{lemma1}, $\mathbf{f} + \mathbf{b} \in S$ and clearly $(\mathbf{f} + \mathbf{b}) - \mathbf{b} = \mathbf{f} \not\in \mathrm{pos}(S) \setminus S.$ Thus, $\mathbf{f} + \mathbf{b} \in \mathrm{Ap}(S,\mathbf{b})$. Now, suppose that there exist $\mathbf a \in \mathrm{Ap}(S,\mathbf{b})$ such that $\mathbf f + \mathbf b \prec \mathbf a$. In this case, $\mathbf{a} - \mathbf{b} \in \mathcal{H}(S)$, so $\mathbf{a} - \mathbf{b} \preceq \mathbf{f}$ and therefore $\mathbf{a} \prec \mathbf{f} + \mathbf{b}$ which contradicts the anti-symmetry property of $\prec.$
\end{proof}

Next result generalizes \cite[Proposition 2.20]{ns-book}.

\begin{proposition}\label{Prop ApPF}
Let $S$ be a submonoid of $\mathbb{N}^d$ and $\mathbf{b} \in S \setminus \{0\}$. Then $\mathrm{PF}(S) \neq \varnothing$ if and only if $\mathrm{maximals}_{\preceq_S} \mathrm{Ap}(S,\mathbf{b}) \neq \varnothing$. In this case, \begin{equation}\label{ecu1} \mathrm{PF}(S) = \{\mathbf{a} - \mathbf{b} \mid \mathbf{a} \in \mathrm{maximals}_{\preceq_S} \mathrm{Ap}(S,\mathbf{b}) \}.\end{equation}
\end{proposition}

\begin{proof}
Suppose that there exists $\mathbf{b}' \in \mathrm{PF}(S)$ and let $\mathbf{a} = \mathbf{b}' + \mathbf{b}$. Clearly $\mathbf{a} \in \mathrm{Ap}(S,\mathbf{b})$, and we claim that $\mathbf{a} \in \mathrm{maximals}_{\preceq_S} \mathrm{Ap}(S,\mathbf{b})$. Otherwise, there exists $\mathbf{a}' \in
\mathrm{Ap}(S,\mathbf{b})$ such that $\mathbf{a'} - \mathbf{a} \in S$, then \[\mathbf{a'} - \mathbf{b} = (\mathbf{a'} - \mathbf{b}) + \mathbf{b}' - \mathbf{b}'= \mathbf{b}' + (\mathbf{a'} - (\mathbf{b}' + \mathbf{b})) = \mathbf{b}' + (\mathbf{a'} - \mathbf{a}) \in S,\] which contradicts the definition of Ap\'ery set of $S$ relative to $\mathbf{b}$. Therefore, $\mathbf{b}' = \mathbf{a} - \mathbf{b}$ with $\mathbf{a} \in \mathrm{maximals}_{\preceq_S} \mathrm{Ap}(S,\mathbf{b})$. Consider now $\mathbf{a}'' \in \mathrm{maximals}_{\preceq_S} \mathrm{Ap}(S,\mathbf{b})$ and let $\mathbf{b}'' = \mathbf{a}''-\mathbf{b}$. If $\mathbf{b}'' \not\in \mathrm{PF}(S),$ then $\mathbf{b}'' + \mathbf{a}_i \not\in S$, for some $i \in \{1, \ldots, n\}$, that is to say, $\mathbf{a}''+\mathbf{a}_i \in \mathrm{Ap}(S,\mathbf{b})$ which is not possible by the maximality of $\mathbf{a}''$ in $\mathrm{Ap}(S,\mathbf{b})$ with respect to $\preceq_S$.
\end{proof}

Observe that \eqref{ecu1} holds for every MPD-semigroup.

We end this section by proving that MPD-semigroups are stable by gluing. First of all, let us recall the notion of gluing of affine semigroups.

Given an affine semigroup $S\subseteq \mathbb N^d$, denote by $G(S)$ the group spanned by $S$, that is,
\[G(S)=\big\{ \mathbf z \in \mathbb Z^m \mid \mathbf z= \mathbf a - \mathbf b, \mathbf a, \mathbf b\in S \big\}.\]

Assume that $S$ is finitely generated. Let $\mathcal{A}$ be the minimal generating system of $S$ and $\mathcal A= \mathcal A_1\cup \mathcal A_2$ be a nontrivial partition of $\mathcal A$. Let $S_i$ be the submonoid of $\mathbb{N}^d$ generated by $A_i,\ i\in \{1,2\}$. Then $S=S_1+S_2$. We say that $S$ is the \textbf{gluing} of $S_1$ and $S_2$ by $\mathbf d$ if
\begin{itemize}
\item $\mathbf d\in S_1\cap S_2$ and,
\item $G(S_1)\cap G(S_2) = \mathbf d\mathbb Z$.
\end{itemize}
We will denote this fact by $S=S_1+_{\mathbf d} S_2$.

\begin{theorem}\label{frob-gluing}
Let $S$ be an finitely generated submonoid of $\mathbb{N}^d$.
Assume that $S=S_1+_\mathbf d S_2$. If $S_1$ and $S_2$ are MPD-semigroups, and $\mathbf{b}_i \in \mathrm{PF}(S_i),\ i = 1,2,$ then $\mathbf b_1 + \mathbf b_2 + \mathbf d \in \mathrm{PF}(S)$. In particular, $S$ is a MPD-semigroup.
\end{theorem}

\begin{proof}
Let $\mathbf{b} := \mathbf b_1 + \mathbf b_2 + \mathbf d$. Since $\mathbf{b}_i \in \mathrm{pos}(S_i),\ i = 1,2,$ and $\mathbf{d} \in S_1 \cap S_2,$ we conclude that $\mathbf{b} \in \mathrm{pos}(S)$. If $\mathbf{b} \in S$, then there exist $\mathbf{b}'_i \in S_i,\ i = 1,2,$ such that $\mathbf{b} = \mathbf{b}'_1 + \mathbf{b}'_2$. Then $\mathbf{b}_1 + \mathbf{d} - \mathbf{b}'_1 = \mathbf{b}'_2 - \mathbf{b}_2 \in G(S_1)\cap G(S_2) = \mathbf d\mathbb Z$. So, there exist $k \in \mathbb{Z}$ such that $\mathbf{b}_1 + \mathbf{d} - \mathbf{b}'_1 = \mathbf{b}'_2 - \mathbf{b}_2 = k \mathbf d$. If $k \leq 0,$ then $\mathbf{b}_2 = \mathbf{b}'_2 - k \mathbf d \in S_2$, which is impossible. If $k > 0,$ then $\mathbf{b}_1 = \mathbf{b}'_1 + (k-1) \mathbf d \in S_1$, which is also impossible. All this prove that $\mathbf{b} \in \mathcal{H}(S)$.

Now, let $\mathbf{a} \in S \setminus \{0\}$. Again there exist $\mathbf{b}'_i \in S_i,\ i = 1,2,$ such that $\mathbf{a} = \mathbf{b}'_1 + \mathbf{b}'_2$. Since $\mathbf d \in S_1 \cap S_2 \subset S$. We have that $\mathbf{b}_1 + \mathbf{b}'_1 \in S_1$ and $\mathbf{b}_2 + \mathbf{b}'_2 + \mathbf{d} \in S_2$. Thus $\mathbf{b} + \mathbf{a} \in S$, and we are done.
\end{proof}

\begin{example}
Let $S_1=\{(x,y,z)\in\mathbb N^3\mid z=0\}\setminus\{(1,0,0)\}$ and $S_2=\{(x,y,z)\in\mathbb N^3\mid x=y\}\setminus\{(0,0,1)\}$. Clearly, $(1,0,0)\in \mathrm{PF}(S_1)$ and $(0,0,1)\in \mathrm{PF}(S_2)$. They are minimally generated by $\{ (2,0,0),(3,0,0),(0,1,0),(1,1,0)\}$  and $\{ (1,1,0),(1,1,1),(0,0,2),(0,0,3)\}$, respectively. The set $G(S_1)\cap G(S_2)$ is equal to $(1,1,0)\mathbb{Z}$ and $S_1+S_2$ is generated by \[\{(2,0,0),(3,0,0),(0,1,0),(1,1,0),(1,1,1),(0,0,2),(0,0,3)\}\] By Theorem \ref{frob-gluing},  $(1,0,0)+(0,0,1)+(1,1,0)=(2,1,1)$ belongs to $\mathrm{PF}(S_1+S_2) $.
\end{example}

\section{On the irreducibility of MPD-semigroups}\label{Sect Irr}

Now, let us study the irreducibility of MPD-semigroups with special emphasis in the $\mathcal C-$semigroups case. Recall that a submonoid of $\mathbb{N}^d$ is \textbf{irreducible} if cannot be expressed as an intersection of two submonoids of $\mathbb{N}^n$ containing it properly.

\begin{lemma}\label{LemaIrr}
Let $\mathbf a \in \mathrm{PF}(S)$. If $2\, \mathbf{a} \in S$, then $S \cup \{\mathbf{a}\}$ is the submonoid of $\mathbb{N}^d$ generated by $\mathcal{A} \cup \{\mathbf{a}\}$. Moreover,
\begin{enumerate}[(a)]
\item if $\mathbf{a} \in \mathrm{F}(S)$ and $\mathrm{PF}(S) \neq \{\mathbf a\}$ then $S \cup \{\mathbf{a}\}$ is a MPD-semigroup;
\item if $S$ is a $\mathcal{C}-$semigroup, then $S \cup \{\mathbf{a}\}$ is a $\mathcal{C}-$semigroup.
\end{enumerate}
\end{lemma}

\begin{proof}
By definition, $\mathbf{a} + \mathbf{b} \in S \subset S \cup \{\mathbf{a}\} $, for every $\mathbf{b} \in S$, and by hypothesis $2\, \mathbf{a} \in S$, so $k\, \mathbf{a} \in S$ for every $k \in \mathbb{N}$; thus, $S \cup \{\mathbf{a}\}$ is the submonoid of $\mathbb{N}^d$ generated by $\mathcal{A} \cup \{\mathbf{a}\}$.

Suppose now that $\mathbf{a}$ is a Frobenius element of $S$ and that $\mathrm{PF}(S) \neq \{\mathbf a\}$. Let $\mathbf{b} \in \mathrm{PF}(S) \setminus \{\mathbf{a}\}$. Clearly, $\mathbf{a} + \mathbf{b} \in S$, because $\mathbf{a} + \mathbf{b}$ is greater than $\mathbf{a}$ for every term order on $\mathbb{N}^d$, therefore $\mathbf{b} + (S \cup \{\mathbf a\}) \setminus \{0\} \subset S \subset S \cup \{\mathbf a\}$ and we conclude that $\mathrm{PF}(S \cup \{\mathbf a\}) \neq \varnothing$ which proves (a).

Finally, since $\mathbf a \in \mathcal{H}(S)$, we have that $\mathrm{pos}(S \cup \{\mathbf{a}\}) = \mathrm{pos}(S)$. Therefore $\mathcal{H}(S \cup \{\mathbf{a}\}) \subset \mathcal{H}(S)$, that is, $S \cup \{\mathbf{a}\}$ is a $\mathcal{C}-$semigroup if $S$ it so as claimed in (b).
\end{proof}

\begin{proposition}\label{PropIrr?}
If $S$ has a Frobenius element, $\mathbf{f}$, and is irreducible then
\begin{enumerate}[(a)]
\item $S$ is maximal among all the submonoids of $\mathbb{N}^d$ having $\mathbf{f}$ as a Frobenius element.
\item $S$ has an unique Frobenius element.
\end{enumerate}
\end{proposition}

\begin{proof}
Suppose that $S$ is irreducible and let $S'$ be a submonoid of $\mathbb{N}^d$ having $\mathbf{f}$ as Frobenius element. Since, by Lemma \ref{LemaIrr}, $S \cup \{\mathbf{f}\}$ is a MPD-semigroup and $S = (S \cup \{\mathbf{f}\}) \cap S'$, we conclude that $S = S'$. Finally, if $S$ has two Frobenius elements, say $\mathbf{f}_1$ and $\mathbf{f}_2$, then $S = (S \cup \{\mathbf{f}_1\}) \cap (S \cup \{\mathbf{f}_2\})$ which contradicts the irreducibility of $S$.
\end{proof}

Notice that the submonoids in condition (a) are necessarily MPD-semi\-groups; indeed the existence of a Frobenius elements in a submonoid of $\mathbb{N}^d$ implies that the submonoid is a MPD-semigroup by Lemma \ref{lemma1}.

\begin{theorem}\label{Th Irr}
If $S$ has a Frobenius element, $\mathbf{f}$, and is irreducible, then either $\mathrm{PF}(S) = \{\mathbf{f}\}$ or $\mathrm{PF}(S) = \{\mathbf{f}, \mathbf{f}/2\}$.
\end{theorem}

\begin{proof}
Suppose that $\mathrm{PF}(S) \neq \{\mathbf{f}\}$. Now, since $\mathrm{PF}(S)$ has cardinality greater than or equal to two, there exists $\mathbf{a} \in \mathrm{PF}(S)$ different from $\mathbf{f}$. If $2\mathbf{a} \in S$, then, by Lemma \ref{LemaIrr}, $S \cup \{\mathbf{a}\}$ are $S \cup \{\mathbf{f}\}$ are a submonoid of $\mathbb{N}^d$ whose intersection is $S$, in contradiction with irreducibility of $S$. Therefore, we may assume that $2\mathbf{a} \not\in S$ which implies $2 \mathbf{a} \in \mathrm{PF}(S)$. Then $\mathbf{f} + \mathbf{u} = 2 \mathbf{a}$ for some $\mathbf{u} \in \mathbb{N}^d$, because $\mathbf{f}$ is greater than or equal to $\mathbf{b}$ for every term order on $\mathbb{N}^d$. Notice that $4 \mathbf{a} = \mathbf{f} + (\mathbf{f} + 2 \mathbf{u} ) \in S$; so, by Lemma \ref{LemaIrr}, $\mathbf{S} \cup \{2 \mathbf{a}\}$ is a MPD-semigroup. Now, if $\mathbf{u} \neq 0,$ then $S = (\mathbf{S} \cup \{2 \mathbf{a}\}) \cap (\mathbf{S} \cup \{\mathbf{f}\}),$ in contradiction with irreducibility of $S$. Therefore $2 \mathbf{a} = \mathbf{f}$ and we are done.
\end{proof}

\begin{example}
Let $S$ be the MPD-semigroup of Example \ref{Ex Irr?}. Let us see that $S$ is irreducible. If $S$ is not irreducible, there exist two submonoids, $S_1$ and $S_2$,  of $\mathbb{N}^2$ such that $S = S_1 \cap S_2$. Since $\mathrm{pos}(S) = \mathbb{N}^2$, we have that $\mathrm{pos}(S_1) = \mathrm{pos}(S_2) = \mathbb{N}^2$ and it follows that $\mathcal{H}(S_i) \subseteq \mathcal{H}(S),\ i = 1,2$. On other hand, since $(7,2) \not\in S$, then $(7,2) \not\in S_1$ or $(7,2) \notin S_2$. Therefore, $S_1$ or $S_2$ is a submonoid of $\mathbb{N}^2$
such that $\mathbf{f} \in \mathrm{F}(S_1)$ or $\mathbf{f} \in \mathrm{F}(S_2)$, respectively. Now, by Proposition \ref{PropIrr?}, we conclude that $S = S_1$ or $S=S_2$, that is, $S$ is irreducible.
\end{example}

If $n=1$, the converse Theorem \ref{Th Irr} is also true (see \cite[Section 4.1]{ns-book}). Let us see that this is also happen if we fix the cone.

\begin{definition}
If $S$ is $\mathcal{C}-$semigroup, we say that $S$ is \textbf{$\mathcal{C}-$irreducible} if cannot be expressed as an intersection of two finitely generated submonoids $S_1$ and $S_2$ of $\mathbb{N}^d$ with $\mathrm{pos}(S_1) = \mathrm{pos}(S_2) = \mathrm{pos}(S)$ containing it properly.
\end{definition}

\begin{proposition}\label{Prop IrredC}
If $S$ is a $\mathcal{C}-$semigroup such that $\mathrm{PF}(S) = \{\mathbf{f}\}$ or
$\mathrm{PF}(S) = \{\mathbf{f}, \mathbf{f}/2\}$,
then $S$ is $\mathcal{C}-$irreducible.
\end{proposition}

\begin{proof}
Suppose there exist two finitely generated submonoids $S_1$ and $S_2$ of $\mathbb{N}^d$ with  $\mathrm{pos}(S_1) = \mathrm{pos}(S_2) = \mathrm{pos}(S)$ such that $S = S_1 \cap S_2$; in particular, $S_1$ and $S_2$ are $\mathcal{C}-$semigroups. For $i=1,2,$ we take $\mathbf b_i \in \mathrm{maximals}_{\preceq_S} S_i \setminus S.$ Since $S_i \setminus S$ is finite, $\mathbf b_i$ is well-defined for $i = 1,2$. By maximality, $\mathbf b_i + \mathbf{a} \in S$ for every $\mathbf{a} \in S \setminus \{0\},\ i = 1,2$, that is to say, $\mathbf b_i \in \mathrm{PF}(S).$ Therefore, $\mathbf{b}_i = \mathbf{f},\ i = 1,2$ or $\mathbf{b}_i = \mathbf{f}$ and $\mathbf{f}_j = \mathbf{f}/2,\ \{i,j\} = \{1,2\}$. In the first case, we obtaim $\mathbf{b}_1 = \mathbf{b}_2$ which is not possible because $\mathbf{b}_i \not\in S,\ i = 1,2$. In the second case, we obtain that $\mathbf{f} \in S_1 \cap S_2 = S$ which obviously is impossible. Therefore, $\mathbf{b}_1$ or $\mathbf{b}_2$ does not exist and we conclude that $S = S_1$ or $S = S_2$.
\end{proof}

\section{PI-monoids}

Let $\preceq_{\mathbb{N}^d}$ be the usual partial order in $\mathbb{N}^d$, that is, $\mathbf{a} = (a_1, \ldots, a_d) \preceq_{\mathbb{N}^d} \mathbf{b} = (b_1, \ldots, b_d)$ if and only if $a_i \leq b_i,\ i \in \{1, \ldots, d\}$.

\begin{definition}
If $S$ is a submonoid of $\mathbb{N}^d$, we define the \textbf{multiplicity} of $S$ as $m(S) := inf_{\preceq_{\mathbb{N}^d}} (S \setminus \{0\})$.
\end{definition}

If $d=1$, the notion of multiplicity defined above agrees with the notion of multipliciy of a numerical semigroup (see \cite[Section 2.2]{ns-book}) Let us introduce a new family of submonoids of $\mathbb{N}^d$, that we have called \textbf{principal ideal monoids}, or $\mathrm{PI}-$monoids for short. This family generalizes the notion of MED-semigroups (see \cite[Chapter 3]{ns-book} for $d > 1$.

\begin{definition}
A submonoid $S$ of $\mathbb{N}^d$ is said to be a \textbf{PI-monoid} if there exist a submonoid $T$ of $\mathbb{N}^d$ and $\mathbf a \in T \setminus \{0\}$ such that $S = \big(\mathbf a\ + T \big) \cup \{0\}$.
\end{definition}

Clearly, PI-monoids are not always affine semigroups, since they are not necessarily finitely generated. We will explicitly provide a minimal generating system of any PI-monoid later on, first let us explore some its properties.

\begin{example}\label{exPI}
In $\mathbf N^2$, an example of finitely generated PI-monoid is $S_1=(2,2)+\langle  (1,1) \rangle=\langle (2,2),(3,3)\rangle$.

To obtain a non-finitely generated PI-monoid of $\mathbf N^2$, consider $T=\mathbf N^2$ and $a=(1,1)$. The PI-monoid $S_2=(1,1)+\mathbf N^2$ is equal to $\{(x,y)\mid x\geq 1,~ y\geq 1\}\cup \{(0,0)\}$ which is not a finitely generated submonoid of $\mathbf N^2$.
\end{example}

\begin{lemma}\label{lemma1n}
If $S \subseteq \mathbb{N}^d$ is a PI-monoid, then $m(S) \in S \setminus \{0\}.$ In particular, $m(S) = \min_{\preceq_{\mathbb{N}^d}} (S \setminus \{0\}).$
\end{lemma}

\begin{proof}
Since $S$ is a PI-monoid, there exist a submonoid $T$ of $\mathbb{N}^d$ and $\mathbf a \in T \setminus \{0\}$ such that $S = \big(\mathbf a + T \big) \cup \{0\}$. Clearly, $\mathbf a = \min_{\preceq_{\mathbb{N}^d}} (S \setminus \{0\}).$
\end{proof}

The following result is the generalization of \cite[Proposition 3.12]{ns-book}.

\begin{proposition}\label{Prop2}
Let $S$ be a submonoid of $\mathbb{N}^d$. Then, $S$ is a PI-monoid if and only if $m(S) \in S \setminus \{0\}$ and $(S \setminus \{0\}) - m(S)$ is a submonoid of $\mathbb{N}^d$.
\end{proposition}

\begin{proof}
If $S$ is a PI-monoid, by Lemma \ref{lemma1n}, $m(S) = \min_{\preceq_{\mathbb{N}^d}} (S \setminus \{0\});$ moreover, there is a submonoid $T$ of $\mathbb{N}^d$ such that $S = (m(S) + T) \cup \{0\}$. So, $(S \setminus \{0\}) - m(S) = T$ is a submonoid of $\mathbb{N}^d$. For the converse implication, it suffices to note that $S = ( m(S) +T) \cup \{0\}$ and that $T = (S \setminus \{0\}) - m(S)$.
\end{proof}

\begin{corollary}\label{Cor3}
If $S \subseteq \mathbb{N}^d$ is a PI-monoid, then there exist an unique submonoid $T$ of $\mathbb{N}^d$ and an unique $\mathbf a \in T \setminus \{0\}$ such that $S = (\mathbf a + T) \setminus \{0\}$
\end{corollary}

\begin{proof}
It is clear that $\mathbf a$ must be equal to $m(S)$ and that $T$ must be equal to $(S \setminus \{0\}) - m(S)$.
\end{proof}

\begin{remark}
Given a submonoid $S$ of $\mathbb{N}^d$, we will write $\mathrm{PI}(S)$ for the set \[\big\{ (\mathbf{a} + S) \cup \{0\}\ \mid \mathbf a \in S \setminus \{0\} \big\}.\] Observe that, as an immediate consequence of Corollary \ref{Cor3}, we have that
the set $\{ \mathrm{PI}(S) \mid S\ \text{is a submonoid of}\ \mathbb{N}^d\}$ is a partition of the set of all PI-monoids of $\mathbb{N}^d$. Moreover, if $\mathscr{A}$ denotes the set of all submonoids of $\mathbb{N}^d$, for some $d$, and $\mathscr{P}i$ denotes the set of all PI-monoids of $\mathbb{N}^d$, for some $d$, we have an injective map \[\mathscr{A} \longrightarrow \mathscr{P}i;\ S \mapsto (\min_{lex}(S \setminus \{0\}) + S) \cup \{0\},\] where $lex$ means the lexicographic term order on $\mathbb{N}^d$.
\end{remark}

Recall that a system of generators $\mathcal{A}$ of a submonoid $A$ of $\mathbb{N}^d$ is said to be minimal if no proper subset of $\mathcal{A}$ generates $A$. The following result identifies a minimal system of generators of an PI-monoid.

\begin{proposition}\label{Prop8}
Let $S$ be a submonoid of $\mathbb{N}^d$. Then $S$ is a PI-monoid if and only if \[\big(\mathrm{Ap}(S,m(S)) \setminus \{0\}\big) \cup \{m(S)\}\] is a minimal system of generators of $S$.
\end{proposition}

\begin{proof}
By Lemma \ref{lemma1n}, if $S$ is a PI-monoid, then $m(S) \in S \setminus \{0\}$ Moreover, by Proposition \ref{Prop6-Cor7}, we have that $\mathcal{A} := \big(\mathrm{Ap}(S,m(S)) \setminus \{0\}\big) \cup \{m(S)\}$ is a system of generators of $S$. So, it suffices to prove that $\mathcal{A}$ is minimal. Let us assume the contrary, that is, there exists $\mathbf a \in \mathcal A$ such that $\mathcal{A} \setminus \{\mathbf a\}$ generates $S$. By the minimality of $m(S),\ \mathbf a \neq m(S)$. Thus, $\mathbf a \in \mathrm{Ap}(S,m(S)) \setminus \{0\}$ and there exists $\mathbf b$ and $\mathbf c \in S$ with $\mathbf a = \mathbf b + \mathbf c$. By Proposition \ref{Prop2}, we know that $\mathbf b - m(S) + \mathbf c - m(S) = \mathbf d - m(S)$ for some $\mathbf d \in S \setminus \{0\}$. Therefore, $\mathbf a = \mathbf d + m(S) \not\in \mathrm{Ap}(S,m(S))$, which is impossible. Conversely, if $S$ is not a PI-monoid, by Proposition \ref{Prop2}, we have that $(S \setminus \{0\}) - m(S)$ is not a submonoid of $\mathbb{N}^d$. So, there exists $\mathbf a$ and $\mathbf b \in \mathrm{Ap}(S,m(S)) \setminus \{0\}$ such that $\mathbf a - m(S) + \mathbf b - m(S) \not\in (S \setminus \{0\}) - m(S)$. In particular, $\mathbf a + \mathbf b - m(S) \not\in S$ and consequently $\mathbf a + \mathbf b \in \mathrm{Ap}(S,m(S)).$ So, $\mathrm{Ap}(S,m(S))$ is not a minimal system of generators of $S$.
\end{proof}

Now, we will show that PI-monoids have non-trivial infinite pseudo-Frobenius set. Recall that every submonoid $S$ of $\mathbb{N}^d$ defines a natural partial order on $\mathbb{N}^d$ as follows: $\mathbf x \preceq \mathbf y$ if and only if $\mathbf y - \mathbf x \in S$. As in the previous section this partial order will be denoted as $\preceq_S$.

\begin{corollary}\label{Cor10}
A submonoid $S$ of $\mathbb{N}^d$ is a PI-monoid if and only if $m(S) \in S \setminus \{0\}$ and $\mathrm{Ap}(S,m(S)) \setminus \{0\} = m(S) + \mathrm{PF}(S)$.
\end{corollary}

\begin{proof}
If $S$ is a PI-monoid, then $m(S) \in S$ by Lemma \ref{lemma1n} and, by Proposition \ref{Prop8}, $\big(\mathrm{Ap}(S,m(S)) \setminus \{0\}\big) \cup \{m(S)\}$ is a minimal system of generators of $S$. Therefore, $\mathrm{Ap}(S,m(S)) \setminus \{0\} = \mathrm{maximals}_{\preceq_S}(\mathrm{Ap}(S,m(S))$. Now, by Proposition \ref{Prop ApPF}, we are done. Conversely, let us suppose that $m(S) \in S$ and that $\mathrm{Ap}(S,m(S)) \setminus \{0\} = m(S) + \mathrm{PF}(S)$. By Proposition \ref{Prop ApPF}, we have that $\mathrm{PF}(S) = \mathrm{maximals}_{\preceq_S}(\mathrm{Ap}(S,m(S)) - m(S)$. Therefore, $ \mathrm{Ap}(S,m(S)) \setminus \{0\} = \mathrm{maximals}_{\preceq_S}(\mathrm{Ap}(S,m(S))$, that is, $\big( \mathrm{Ap}(S,m(S)) \setminus \{0\} \big) \cup \{m(S)\}$ is a minimal system of generators of $S$. Now, by Proposition \ref{Prop8}, we are done.
\end{proof}

Putting all this together, we have the following characterization of the PI-monoids.

\begin{theorem}\label{Th11}
Let $S$ be a submonoid of $\mathbb{N}^d$. The following conditions are equivalent:
\begin{enumerate}
\item $S$ is a PI-monoid.
\item $m(S) \in S \setminus \{0\}$ and $(S \setminus \{0\}) - m(S)$ is closed under addition.
\item $\big(\mathrm{Ap}(S,m(S)) \setminus \{0\}\big) \cup \{m(S)\}$ is a minimal system of generators of $S$.
\item $\{m(S) + \mathrm{PF}(S)\} \cup \{m(S)\}$ is a minimal system of generators of $S$.
\end{enumerate}
\end{theorem}

\begin{example}
	Let $S_1$ and $S_2$ be the PI-monoids of Example \ref{exPI}. For $S_1$ we have $m(S_1)=\{(2,2)\}$ and  $\mathrm{Ap}(S_1,(2,2))=\{(0,0),(3,3)\}$ obtaining that $\{(2,2),(3,3)\}$ is a system of generators of $S_1$ and that $\mathrm{PF}(S_1)=\{(1,1)\}$.
	For $S_2$, $m(S_2)=(1,1)$ and
	$\mathrm{Ap}(S_2,(1,1))=\{(0,0)\}\cup \{(x,1)\mid x\in\mathbf N \setminus\{0,1\}\} \cup \{ (1,y)\mid y\in\mathbf N \setminus\{0,1\} \}$. So $\{(1,1)\} \cup \{(x,1)\mid x\in\mathbf N \setminus\{0,1\}\} \cup \{ (1,y)\mid y\in\mathbf N \setminus\{0,1\}\}$ is a non-finite system of generators of $S_2$ and $\mathrm{PF}(S_2)=\{(x,0)\mid x\in \mathbf N\setminus\{0\}\} \cup \{(0,y)\mid y\in \mathbf N\setminus\{0\}\}$.
\end{example}

Finally, our last results state the relationship between PI-monoids and MPD-semigroups.

\begin{corollary}\label{Cor last1}
Let $S$ be a PI-monoid. Then $S$ is an affine semigroup if and only if $S$ is an MPD-semigroup. In this case, $\mathrm{Ap}(S,m(S))$ is finite.
\end{corollary}

\begin{proof}
Let $S$ be a finitely generated PI-monoid. By Corollary \ref{Cor10}, $\mathrm{PF}(S) \neq \varnothing$. So, by Theorem \ref{Th1}, $S$ is a MPD-semigroup. The converse is trivial as MPD-semigroups are affine semigroups by definition. The last part is a direct consequence of Theorem \ref{Th1} and Corollary \ref{Cor10}.
\end{proof}

The following result was inspired by \cite[Lemma 2.2]{ADT}:

\begin{corollary}
Let $S$ be a PI-monoid. Then there exist a direct system $(S_\lambda, i_{\lambda \mu})$ of MPD-semigroups contained in $S$ such that $S = \varinjlim_{\lambda \in \Lambda} S_\lambda$, where $i_{\lambda \mu} : S_\lambda \to S_\mu$ is the inclusion map.
\end{corollary}

\begin{proof}
Let $\Lambda = \left\{ \lambda \subset \{m(S) + \mathrm{PF}(S)\}\ \mid\ \lambda\ \text{is finite} \right\},$ partially ordered by inclusion, and define $S_\lambda$ to be the affine semigroup generated by $\lambda \cup \{m(S)\}$. Clearly, we have that $S_\lambda \subseteq S_\mu$ if $\lambda \subseteq \mu$; in this case, let $i_{\lambda \mu} : S_\lambda \to S_\mu$ is the inclusion map. Now, since $S_\lambda \subseteq S$ for every $\lambda \in \Lambda$, we conclude that $S = \varinjlim_{\lambda \in \Lambda} S_\lambda$ by Theorem \ref{Th11}, because
$\{m(S) + \mathrm{PF}(S)\} \cup \{m(S)\}$ is a minimal system of generators of $S$.

Finally, let us see that $S_\lambda$ is a MPD-semigroup for every $\lambda \in \Lambda$. To do that, we first note that $m(S_\lambda) = m(S) \in S_\lambda$, for every $\lambda \in \Lambda$. Let $\mathcal{A} = \{\mathbf{a}_1, \ldots, \mathbf{a}_n\} \subseteq \mathrm{PF}(S)$ and let $\lambda = \{m(S) + \mathcal{A}\} \in \Lambda$. Then, $\mathrm{Ap}(S_\lambda,m(S)) = \{0, \mathbf{a}_1, \ldots, \mathbf{a}_n\}$ is finite, in particular, $\mathrm{maximals}_{\preceq_S}(\mathrm{Ap}(S,m(S)) - m(S)$ is a non-empty finite set. Therefore, by Proposition \ref{Prop ApPF}, $\mathrm{PF}(S) \neq \varnothing$, that is to say, $S_\lambda$ is a MPD-semigroup.
\end{proof}

\medskip
\noindent\textbf{Acknowledgement.}
This paper was originally motivated by a question made of Antonio Campillo and F\'elix Delgado about $\mathcal{C}-$semigroups during a talk of the fourth author at the GAS seminar of the SINGACOM research group. The question is answered in a wider context by Corollary \ref{cor1}. Part of this paper was written during a visit of the second author to the Universidad de C\'adiz (Spain) and to the IEMath-GR (Universidad de Granada, Spain), he thanks these institutions for their warm hospitality. The authors would like to thank Antonio Campillo, F\'elix Delgado and Pedro A. Garc\'{\i}a-S\'anchez for usseful suggestions and comments.


\end{document}